\documentclass[10pt]{amsart}

\usepackage{enumerate}
\usepackage{amsmath, amsthm, amssymb, latexsym, graphics}
\usepackage{parskip}
\usepackage[francais,english]{babel}
\usepackage[utf8]{inputenc}

\usepackage{enumerate}

\newcommand{\F}{\mathbb{F}}
\newcommand{\N}{\mathbb{N}}

\newcommand{\Q}{\mathbb{Q}}

\newcommand{\jumpm}[2]{\operatorname{Jump}_{#1}(#2)}

\renewcommand {\epsilon}{\varepsilon}

\renewcommand {\leq}{\leqslant}
\renewcommand {\geq}{\geqslant}

\newcommand{\ssi}{\Leftrightarrow}

\newcommand{\ppcm}{\text{ppcm}}

\newcommand{\tq}{\;\: | \:\:}
\newcommand{\mesp}{\;\:}

	\newcommand {\vfi}{\varphi}

	\theoremstyle{plain}
	\newtheorem{thm}{Th\'{e}or\`{e}me}[section]
	\newtheorem{lem}[thm]{Lemme}
	\newtheorem{prop}[thm]{Proposition}
	\newtheorem{corr}[thm]{Corollaire}
	
	\theoremstyle{definition}
	\newtheorem{defn}[thm]{D\'{e}finition}
	\newtheorem{exemple}[thm]{Exemple}

	\theoremstyle{remark}
	\newtheorem{remarque}[thm]{Remarque}
	
	\newtheorem{nota}[thm]{Notation}
	\newtheorem{fait}[thm]{Fait}

\AtBeginDocument{%
	
	{}}

\AtEndDocument{\bigskip{\footnotesize%
		\textsc{ Instit\"{u}t f\"{u}r Mathematishe Logik und
			Grundladenforchung Fachbereich und Informmatik 
			Einsteinstra{\ss}e 62, 48149 M\"{u}nster, Deutschland
		} \par 
		\textit{E-mail address}, G.Onay: \texttt{onay@uni-muenster.de} }}

\usepackage[top=3cm, bottom=2.5cm, right=3cm, left=3cm]{geometry}

\usepackage{hyperref}
\hypersetup{
colorlinks   = true}

\author{Gönenç Onay }
\thanks{Research partially funded by DFG through SFB 878.}
\makeatletter
\def\blfootnote{\gdef\@thefnmark{}\@footnotetext}
\makeatother
\title{Modules $C$-minimaux sur des anneaux polyn\^omes tordus}
\date{}
\begin{document}
	\sloppy

		\maketitle
	\selectlanguage{french}
	
	\begin{abstract}
		Dans cette article on \'etudie les modules munis d'une ultram\'etrique par le point de vue de la notion geom\'etrique, $C$-minimialit\'e. On donne une caract\'erisation compl\`ete des modules valu\'es $C$-minimaux sur des anneaux non-commutatives de la forme $R:=K[t;\vfi]$, o\`u $K$ est un corps, $\vfi$ est un endomorphisme de $K$ et $R$ est le $K$-alg\`ebre engendr\'e par $t$, tel que $at=ta^{\vfi}$ pour $a\in K$.
		
		On en d\'eduit par exemple que l'anneau des s\'eries de Puiseux sur un corps fini $\mathbb{F}$ de caract\'eristique $p>0$, est $C$-minimal comme module valu\'e sur l'anneau $\mathbb{F}[t;x\mapsto x^p]$, l'anneau des polyn\^omes additives sur $\mathbb{F}$. De plus,  tout ultraproduit $\mathcal{K}$, des corps alg\`ebriquement clos valu\'es $\mathcal{K}_{p^n}$ de caract\'eristique  $p>0$, \'equipp\'e chacun du morphisme $x\mapsto x^{p^n}$, selon un ultrafiltre $U$ sur $\{p^n \tq n\in \mathbb{N}, \; \text{et} \; p \; \text{premier}\}$  muni de l'automorphisme  limite, {\it Frobenius non-standard}, i.e., $\sigma_{U}:=\lim_{U} x \mapsto x^{p^n}$, est $C$-minimal comme $\mathcal{K}[t;\sigma]$-module valu\'e.
	\end{abstract}
	
	\selectlanguage{english}
	\begin{abstract}
In this article we study modules endowed with a ultrametric, from the point of view of the geometric notion $C$-minimality. We give a complete characterization of 
$C$-minimal valued modules over non-commutative rings of skew polynomials of the form $R:=K[t;\vfi]$, where $K$ is a field, $\vfi$ an endomorphism of $K$ and
$R$ is the $K$-algebra generated by $t$, such that $at=ta^{\vfi}$ for $a\in K$.

We deduce for instance that the ring of Puiseux series over a finite field $\mathbb{F}$ of characteristic $p>0$, as a valued module over $\mathbb{F}[t;x\mapsto x^p]$ is
$C$-minimal. Moreover, any ultraproduct  $\mathcal{K}$, of 
algebraically closed valued fields $\mathcal{K}_{p^n}$ of characteristic $p>0$, endowed each with the  morphism  $x\mapsto x^{p^n}$,  following a ultrafilter $U$ over  $\{p^n \tq n\in \mathbb{N}, \, \text{et} \; p \; \text{prime}\}$, equipped with the {\it non-standart Frobenius}, i.e., the map
$\sigma_{U}:=\lim_{U} x \mapsto x^{p^n}$, is  $C$-minimal as a $\mathcal{K}[t;\sigma]$-valued module.
	\end{abstract}

	\selectlanguage{french}
	\section{Introduction}
 Pour nous, une $C$-relation est une relation ternaire satisfaisant les axiomes suivants:
\begin{enumerate}
\item $\forall x,y,z \mesp  { C}(x,y,z)\rightarrow { C}(x,z,y) $,
\item $\forall x,y,z \mesp { C}(x,y,z) \rightarrow \neg { C}(y,x,z)$,
\item $\forall x,y,z,w \mesp  { C}(x,y,z)\rightarrow ({ C}(w,y,z)\vee { C}(x,w,z)),$

\item $\forall x,y \mesp x\neq y \rightarrow { C}(x,y,y)$.
\end{enumerate}

Une $C$-relation interprète toujours un arbre (cf. \cite{Adeleke1998}): pour des chaînes maximales $x,y$ et $z$,  ${ C}(x,y,z)$, exprime que $y$ et $z$  branchent au-dessus de là où $x$ et $y$ (ou $z$) branchent.

On rappelle qu'une
distance ultramétrique, est un triplet $(M,\Gamma, d)$,  o\`u $M$ est un ensemble, $\Gamma$ est un ordre linaire ayant un minimum $0$, et $d:M^2\to \Gamma$, une application surjective, telle que, pour tout $x,y,z \in M$, 
\begin{enumerate}
	\item $d(x,y)=0$ si et seulement si $x=y$,
	\item $d(x,z)\leq \min\{d(x,y),d(y,z)\}$.
\end{enumerate}  

Une distance ultram\'etrique induit la $C$-relation 
$${ C}(x,y,z) \ssi d(x,y)=d(x,z)>d(y,z).$$
Donc un groupe abélien ou un corps valué porte en particulier la $C$-relation canonique en inversant l'in\'egalit\'e ci-dessus:
$${ C}(x,y,z) \ssi v(x-y)=v(x-z)<v(y-z).$$ 
Une valuation triviale, c'est-à-dire  prenant au plus deux valeurs, induit la
$C$-relation triviale: ${ C}(x,y,z) \ssi x\neq y=z$.

\begin{defn}On dit qu'une expansion $\mathcal{U}=(U,{ C}, \dots)$ d'une $C$-relation est $C$-minimale si, pour tout $\mathcal{U}'=(U',{ C}, \dots)$, élémentairement équivalente à $\mathcal{U}$, tout sous-ensemble définissable (avec paramètres) de $U'$, est définissable sans quantificateur en n'utilisant que la relation ternaire $C$.
\end{defn}
 
\begin{remarque} En présence d'une distance ultramétrique, la $C$-minimalité se traduit par le fait que tout sous-ensemble définissable  (de la structure elle-même) est une combinaison booléenne de boules «ouvertes» ou «fermées».
\end{remarque}

   La $C$-minimalité apparaît comme une généralisation commune de la minimalité forte et de l'$o$-minimalité. Par exemple, une structure munie de la $C$-relation triviale est $C$-minimale si et seulement si elle est 
fortement minimale. Plus généralement ces trois notions sont étroitement liées. Notons également que les théories $C$-minimales sont dépendantes, i.e, elles ont la propriété  NIP, comme les théories fortement minimales, $o$-minimales et $P$-minimales (on pourra consulter \cite{Haskell1997} pour les définitions des théories $o$-minimales, $C$-minimales et $P$-minimales, dont  le théorème 4.2.6 et la discussion qui le précède pour le résultat concernant NIP). Noter que si une structure est  $o$-minimale, alors il va de même pour sa théorie complète, i.e. toute structure qui lui est élémentairement équivalente est  $o$-minimale. Cette dernière propriété est fausse pour les structures fortement minimales et par  voie de conséquence pour les structures $C$-minimales.

Le groupe additif de $\Q_{ p}$, muni de sa valuation  $p$-adique,  est $C$-minimal mais le corps $\Q_{p}$ ne l'est pas, par exemple
l'ensemble des carrés n'y est pas définissable sans quantificateur dans le langage des corps valués, donc  dans le langage des corps équipés
de la $C$-relation non plus. De fait, un corps valué est $C$-minimal si et seulement s'il est algébriquement clos (cf. \cite{Haskell1994}).

Dans cette \'etude, on reprend les notations et les d\'efinitions (cf. Section 2 pour un rappel) figurant dans nos articles \cite{Onay2017} et \cite{Onay2018a}. Soit $K$ un corps et $\vfi$ un endomorphisme de $K$. On pose $R:=K[t;\vfi]$, $K$-alg\`ebre engendr\'e par $t$ selon la loi de commutativit\'e 
$$at=ta^{\vfi}$$ pour tout $a\in K$. 
Dans cette article, on  caractérise les $R$-modules valués  $C$-minimaux.

Nos résultats principaux  se résument dans les trois théorèmes 
ci-dessous. 

\begin{thm}\label{thmfrtmin}
Un $R$-module trivialement valué infini $M$ est $C$-minimal (i.e. fortement minimal dans ce cas, justifiant la notation $M$, à la place de $(M,v)$) si et seulement si l'une des conditions ci-dessous est vérifiée:
\begin{enumerate}[$1.$]
\item $M$ est divisible et, pour tout $r \in R\setminus\{0\}$, $\text{ann}_M(r)$ est fini
\item l'idéal annulateur de $M$, $\{q \in R \tq x.q=0, \mesp \forall x \in M\}$, est engendré par un polynôme irr\'eductible $r$, $R/rR$ est un corps et 
$M$ est un $(R/rR)$-espace vectoriel. 
\end{enumerate}
\end{thm}

\begin{thm}\label{thmCmintrivi}
Soit $(M,v)$ un module $K$-trivialement valué infini et non-trivialement valué.  Alors
$(M,v)$ est $C$-minimal si et seulement si $v(M)$ est $o$-minimal dans le langage $L'_V$, pour tout $r \in R\setminus\{0\}$, $\text{ann}_M(r)$ est fini, $(M,v)$ est henselien et si l'une des conditions ci-dessous est vérifiée;
\begin{enumerate}[$1.$]
\item $(M,v)$ est divisible,
\item  $M_{>\theta}.t$ contient une boule d'indice fini et $M=M_{tor}\oplus M_{>\theta}$, où $M_{tor}$ est divisible et fortement minimal s'il n'est pas fini.  
\end{enumerate}
\end{thm}

\begin{thm}\label{thmKnontrivi}
Soit $(M,v)$ un module  valué  non $K$-trivialement. Alors
il est $C$-minimal si et seulement s'il est affinement maximal, résiduellement divisible et tel que $v(M)$ soit une $L_V$-structure $o$-minimale. 
\end{thm}

Cet article est organiz\'ee comme suit. La section 2 rappelle les d\'efinitions introduites dans \cite{Onay2017} et \cite{Onay2018a}. La section 3,5,6 tra\^itent respectivement les cas d'une valuation triviale, d'une valuation non-triviale et $K$-triviale et d'une valuation non $K$-triviale. La section 4
s'agit d'une interlude des r\'esultats g\'en\'eraux sur la $C$-minimialit\'e utilis\'es dans les sections 5 et 6.
\section{Pr\'eliminaires}

Soit $K$ un corps infini, $\vfi$ un endomorphisme
de $K$, et $R:=K[t,\vfi]$. On renvoi le lecteur \`a la section 2 de \cite{Onay2017} pour les g\'en\'eralit\'es sur l'anneau $R$ et les $R$-modules.

Rappelons les faits suivants.

On appelle polyn\^ome un \'el\'ement de  $R$ et donc un mon\^ome  un term de la forme $t^na$, o\`u $a\in K$ et  $n\in \mathbb{N}$.  Cet anneau est euclidien \`a droite, en conséquence tout $r \in R$ s'\'ecrit comme 
$$r=t^ns_1\ldots s_1,$$ o\`u les $s_i$ sont irr\'educibles et s\'eparables (i.e., non divible par $t$). Il s'en suit l'existence du $\ppcm(r,q)$ pour $r,q \in R$. Il vient aussi qu'un $R$-module est divisible si et seulement 
s'il l'est par $t$ et par tout $s$ s\'eparable. 
Par un $R$-module, on entendra un $R$-module \`a droite.
Par $ann_M(r)$ $(r\in R)$ on d\'esigne l'ensemble des \'el\'ements annul\'es par $r$ dans un $R$-module $M$, i.e., l'ensemble 
$$\{x\in M \tq x.r=0\}.$$
 La cl\^oture divisible d'un sous module $A$ dans un $R$-module divisible $M$, est un sous-module $A' \supseteq A$ tel que pour tout $a \in A'$ il existe un $r\in R$, non nul avec $a.r \in A$. Les cl\^otures divisibles de $A$ sont isomorphes au-dessus de $A$ en tant que $R$-modules (cf. Lemma 2.6 (4) \cite{Onay2017} et la discussion suivant la preuve).

  Une $K$-cha\^ine est un ordre lin\'eaire $\Delta$ avec un maximum $\infty$, muni d'une action de $K$ satisfaisant, pour tout $a,b \in K^{\times}$ et $\delta, \gamma \in \Delta\setminus\{\infty\}$, les axiomes suivantes:

  \begin{enumerate}
  	\item $\gamma > \delta \to \gamma \cdot a> \delta \cdot a$,
  	\item $\gamma \cdot ab = (\gamma \cdot a) \cdot b$,
  	\item $\gamma \cdot a > \gamma \to \delta 
  	\cdot a > \delta$,
  	\item $\gamma \cdot (a \pm b)\geq 
  	\min\{\gamma\cdot a, \gamma \cdot b\}$,
  	\item $\gamma \cdot 0= \infty \; ;
  	\gamma \cdot 1 =\gamma\; ; \infty\cdot b=\infty$.
  \end{enumerate}
  
 Cet action induit une valuation $v_K$ sur $K$ (cf. \cite{Onay2018a} Proposition 2.4). 
 \begin{defn}
 	On dit que cette action est $K$-triviale si $v_K$ est la valuation triviale, sinon on dit qu'elle est $K$-non-triviale. 
 \end{defn}
 
 \begin{defn}[$R$-chains]\label{rchains}
 	Une $R$-cha\^ine  $(\Delta, <, \infty, \cdot r_{r \in R})$ est une  $K$-cha\^ine telle que
 	\begin{enumerate}
 		\item $\cdot t$ is strictement croissant $\Delta \setminus \{\infty\}$, et $\infty \cdot t=\infty$,
 		\item $\gamma\cdot ta=(\gamma \cdot t)\cdot a$ et $\gamma\cdot t^na = (\gamma \cdot t^{n-1})\cdot ta$ pour tout $\gamma \in \Delta$, et $a\in K$,
 		\item $\gamma \cdot r = \min_{i} \{\gamma \cdot \mathbf{m}_i\}$, pour tout $r \in R$ non nul, o\`u
 		les $\mathbf{m}_i$ sont les mon\^omes de $r$ et pour tout $\gamma \in \Delta$,
 		\item $\gamma \cdot \mathbf{m}_1 \leq \gamma \cdot \mathbf{m}_2 \to \left( 
 		\forall \delta \mesp (\delta <\gamma \to \delta \cdot 
 		\mathbf{m}_1 < \delta \cdot \mathbf{m}_2)\right)$ pout tout mon\^omes 
 		$\mathbf{m}_1, \mathbf{m}_2 \in R$ tels que $0\leq\deg(\mathbf{m}_2)<\deg(\mathbf{m}_1)$,
 		et pour tout $\gamma
 		\neq \infty$.
 	\end{enumerate}
 \end{defn}

 \begin{defn} Un $R$-module valu\'e est un $R$-module, qui est de plus un groupe ab\'elien valu\'e 
 	$(M,\Delta,v)$ (en particulier $v:M\to \Delta$ est surjective),  o\`u $\Delta$ est une $R$-cha\^ine
 	telle que pour tout $x\in M$, l'on ait
 	\begin{enumerate}
 		\item $v(x.a)=v(x)\cdot a$, for all $a \in K$,
 		\item  $v(x.t)=v(x)\cdot t.$
 	\end{enumerate}
Un $R$-module valu\'e est  dit $K$-trivialement  valu\'e (respectivement $K$-non-trivialement valu\'e) si l'action de $K$ sur $\Delta$ est $K$-triviale (respectivement $K$-non-triviale), i.e., $v(x.a)=v(x)$ pour tout $x\in M$ et $a\in K$. Enfin, $(M,\Delta,v)$ est dit trivialement valu\'e si $\vert v(M) \vert \leq 2$.  
 \end{defn}
 
Il est clair que la d\'efinition ci-dessus d'un $R$-module $K$-trivialement valu\'e co\"incide bien avec la celle donn\'ee dans \cite{Onay2017} (cf. la d\'efinition 3.4). En particulier l'application $x \mapsto x.t$ est injective dans tout $R$-module valu\'e.
 
Soit $(M,\Delta,v)$ un $R$-module valu\'e. 
\begin{nota}
	On rappelle que l'on note  $\theta$ la coupure de $\Delta$, d\'efinie par l'axiome (4) en prenant $\mathbf{m}_1=t$ et $\mathbf{m}_2=1$. Par cons\'equent, on note, dans un $R$-module valu\'e $(M,\ldots)$,  ind\'ependemment du fait que cette coupure soit realis\'ee dans $\Delta$;
	\begin{itemize}
		\item $M_{>\theta}:=\{x\in M \tq v(x.t)>v(x)\}$
		\item $M_{\geq \theta}:=\{x\in M \tq v(x.t)\geq v(x)\}$.
	\end{itemize}
\end{nota} 

Par ailleurs, tout $\gamma \in v(M)\setminus\{\theta, \infty\}$, on note $M_{\geq \gamma}$ (respectivement $M_{>\gamma}$) la boule ferm\'ee (respectiement la boule ouverte) centr\'ee  $0$ et de rayon $\gamma$.
Notez  que si la sous-module de torsion $M_{tor}$  est non triviale alors $v(M_{tor})=\{\theta, \infty\}$ est un sous $R$-module trivialement valu\'e (cf. Lemma 3.10 (3) \cite{Onay2017}). 

\begin{nota}
	Comme dans une $R$-cha\^ine munie d'une action $K$-triviale,
	l'application $\gamma \mapsto  \gamma \cdot r$  est \'egale \`a $\gamma \mapsto \gamma \cdot t^k$, pour un certain $k \in \mathbb{N}$, suivant \cite{Onay2017}, on notera parfois $\tau$,  l'application $\gamma \mapsto   \gamma \cdot t$.  
\end{nota}

\begin{defn}[cf. Definition 3.23 \cite{Onay2017}]
	Soit $(M,\Delta,v)$ un $R$-module $K$-trivialement valu\'e. On dit qu'il est hens\'elien si pour tout $s$ s\'eparable, $x\mapsto x.s$ est une bijection de $M_{>\theta}$ (autrement dit, le sous-module $M_{>\theta}$ est divisible par les polyn\^omes s\'eparables).
\end{defn} 

On peut prendre l'\'enonc\'e ci-dessous comme la d\'efition d'un $R$-module {\it r\'esiduellement divisible et affinement maximal}. 
 
\begin{fait}[reformulation du  Th\'eor\`eme 3.35, \cite{Onay2018a}] $(M,\Delta,v)$
	est r\'esiduelement divisible et affinement maximal si et seulement si pour tout non z\'ero $r\in R$ pour tout $z\in M$, il y a un $y \in M$ tel que 
	\begin{itemize}
		\item $y.r=z$,
		\item $v(y.r)=v(y)\cdot r$.
	\end{itemize}	
\end{fait}

\begin{defn}
	Un \'el\'ement $y \in M$, tel que $v(y.r)=v(y)\cdot r$, est appell\'e {\it r\'egulier} pour $r$.
\end{defn}
 Rappelons que $L_V=\{<,(\cdot r)_{r \in R},
 \infty\}$, dit, le langage des $R$-cha\^ines et $L_V'=L_V\cup\{(R_n)_{n \in \N} \}$, o\`u  les $R_n$ sont des pr\'edicats unaires telles que,  $M \models R_n(\gamma)$ si et seulement si, $\vert M_{\geq \gamma}/M_{>\gamma} \vert\geq n$.
Par la suite, on notera $(M,v)$ pour d\'esigner un $R$-module valu\'e.
\section{Cas fortement minimal}

On consid\`ere les $R$-modules trivialement valu\'es. Notez qu'un $R$-module trivialement valu\'e est $C$-minimale si et seulement s'il est fortement minimal.
\begin{lem}\label{frtmin} 
Si $M$ est un $R$-module fortement minimal infini alors, pour tout $r \in R \setminus \{0\}$,
$${ann}_M(r) \mesp \text{fini} \mesp \ssi (M.r \neq 0) \ssi M.r=M.$$
\begin{proof}
Pour tout $r \in R\setminus\{0\}$, ${ann}_M(r)$ et $M.r$, étant définissables, sont finis ou cofinis; de plus, \'{e}tant des sous-groupes, ils ne peuvent être cofinis propres. Si tous les deux sont finis, en considérant l'application
$M \to M, x \mapsto x.r$, qui a comme noyau $ann_M(r)$, on devrait avoir $M$ fini. Donc la seule possibilité restante est celle qui est énoncée.
\end{proof}
\end{lem}

\begin{lem}\label{generateurI}
Soit $I$ un idéal bilatère non nul de $R$. Si $\vfi^n=1$ pour un certain $n \in \N$, alors il existe $q\in R$ à coefficients dans $Fix(\vfi)$  tel que $I=t^nqR$. Si $\vfi$ n'est pas
d'ordre fini, alors $I$ est de la forme $t^kR$ pour un certain $k \in \N$.
 
\begin{proof}
Voir la proposition 2.2.8 dans \cite{cohn}.
\end{proof}
\end{lem}

\begin{corr}
Soit $(M,v)$ un $R$-module   trivialement valué, infini et fortement minimal. Soit $I=\{r \in R \tq M.r=0\}$. Alors, ou bien  $I=\{0\}$ et dans ce cas $M$ est divisible et les annulateurs 
$\text{ann}_M(r)$ sont finis, ou bien $I\neq \{0\}$ et dans ce cas, $I=rR$ pour un $r\in R$ irr\'eductible et $M$ est un $R/rR$-espace vectoriel.
\begin{proof}
Le cas où $I=\{0\}$ découle du lemme \ref{frtmin}. 
Supposons donc $I\neq \{0\}$. Remarquons que $I$ est un idéal bilatère et que,
par le lemme \ref{generateurI},  il existe
$r \in K_0[t;\vfi]$, o\`u $K_0=Fix(\vfi)$), tel que $I=rR$. Soit $d$ le degré de $r$. Alors pour tout $q$ de degré $<d$ vérifiant $r=qs$ on a $M.q \neq 0$ d'où par \ref{frtmin}, $M.q=M$. Par conséquent $M.r=M.qs=M.s=0$, d'où $s \in I$. Donc
$q \in K$. C'est-à-dire que $r$ est irr\'eductible. Par l'existence du ppcm, et par le théorème de Bézout, le quotient $R/rR$ est un corps.
\end{proof}
\end{corr}
\begin{lem}\label{excorr4.6}
 Si $M$ est un $R$-module fortement minimal, si l'action de $t$ est injective et si $\vfi$ n'est pas d'ordre fini, alors l'idéal $I$ ci-dessus est nul.
\begin{proof}
Si $\vfi$ n'est pas d'ordre fini et si $I\neq \{0\}$, par le lemme \ref{generateurI}, $I$ est engendré par $t^k$, pour un certain $k \in \N_{>0}$. Puisque la multiplication par $t$ est injective ceci n'est pas possible. 
\end{proof}
\end{lem}

{\bf Preuve du théorème \ref{thmfrtmin}}.
Par ce qui précède il suffit de montrer qu'un $R$-module satisfaisant $1$ ou 
$2$ est fortement minimal. Si $M$  est un $R$-module satisfaisant $1$,  alors,  par la proposition 2.15 \cite{Onay2017}, toute partie définissable de $M$ est finie ou cofinie. Si $M$ satisfait $2$, alors 
la structure de $R$-module  sur $M$ coïncide avec sa structure de $R/rR$-espace vectoriel. Or tout espace vectoriel est fortement minimal.

\section{Interlude: conséquences de la $C$-minimalité}

Rappelons que dans un groupe (ou module )valu\'e $G$, on note $G_{\geq \gamma}$ (respectivement $G_{>\gamma}$ la boule ferm\'ee de centre $0$ et de rayon $\gamma$ (respectivement la boule ouverte de centre $0$ et de rayon $\gamma$). On utilisera très fréquemment la proposition ci-dessous.
\begin{prop}\label{c-min_d'indice_fini}
Si $G$ est un groupe valué  ${C}$-minimal alors, pour tout sous-groupe  définissable infini propre $F$ de G, il existe $\gamma \in v(G)$, tel que $F$ contient $G_{\geq \gamma}$ (ou $G_{> \gamma}$), et $G_{\geq \gamma}$ (ou $G_{> \gamma}$) est d'indice fini dans $F$.
\begin{proof}
Voir  \cite{Macpherson1996} ou \cite{Simonetta2001}  1.6(ii).
\end{proof}
\end{prop}

\begin{prop}
Soit $\mathbb{G}=(G,v,\dots)$  un groupe valué, porteur de structure additionnelle, et $C$-minimal. Alors \begin{enumerate}[$1.$]
\item pour tout $\gamma \in v(G)$, le groupe $G_{\geq \gamma}/G_{>\gamma}$ 
muni de la structure induite par $\mathbb{G}$ est fortement minimal,
\item la chaîne $v(G)$ munie de la structure induite par $\mathbb{G}$ est
$o$-minimale.
\end{enumerate} 
\begin{proof}
Voir  \cite{Macpherson1996}.
\end{proof}
\end{prop}

\begin{corr} Soit $(M,v)$ un $R$-module  valué  $C$-minimal.
	Notons $K/v_K$ le corps r\'esiduel associ\'e \`a la valuation $v_K$ induit par l'action $K$ sur $v(M)$. Alors,
\begin{enumerate}
\item pour tout  $\gamma \in v(M)$, $M_{\geq \gamma}/M_{>\gamma}$ est un $(K/v_K)$-espace vectoriel fortement minimal et,
\item la $R$-chaîne $v(M)$ est $o$-minimale.
\end{enumerate}
\end{corr}

\begin{remarque}
Dans le point $1$ ci-dessus, si $(M,v)$ est $K$-trivialement valué, i.e. si $v_K$ est triviale, alors $K/v_K=K$ et donc tout   $M_{\geq \gamma}/M_{>\gamma}$ est un $K$ espace-vectoriel. \end{remarque}

\section{Cas d'une valuation non triviale et $K$-triviale}
Une des implications du théorème \ref{thmCmintrivi}  est contenue  dans la proposition suivante.

\begin{prop}\label{divhens->c-min}
Si $(M,v)$  est un module $K$-trivialement valué divisible henselien, si pour tout $q \in R\setminus\{0\}$, $\text{ann}_M(q)$ est fini  et si
$v(M)$ est ${o}$-minimal dans le langage $L'_V$, alors $(M,v)$ est ${ C}$-minimal.
\begin{proof}
Noter qu'un tel $(M,v)$ est r\'esiduellement divisible et affinement maximal par le corollaire 3.29, \cite{Onay2017}.  Nous renvoyons donc à la preuve de la proposition \ref{cmin}  
dont la présente proposition n'est qu'un cas particulier. 
\end{proof}
\end{prop}

\begin{lem}\label{puiseux-stayla-c_min}
Soit $(M,v)$ un module $K$-trivialement valué, tel que 
\begin{enumerate}
\item $M_{tor}$ est fini et $M=M_{tor} \oplus M_{>\theta}$, et
\item $M_{>\theta}$ est divisible et $v(M)$ est $o$-minimal.
\end{enumerate}
Alors $(M,v)$ est $C$-minimal.
\begin{proof} Si $M_{tor}=0$ alors, $M=M_{>\theta}$ est  divisible henselien donc $C$-minimal par la proposition précédante. Sinon, soit $N'$ une clôture divisible de $M_{tor}$. Posons $N:=N'\oplus M_{>\theta}$ et on étend $v$ à $N$ en posant $v(x_{tor} + x_{>\theta})=\theta$ si $x_{tor}$ est non nul. D'où $(N,v)$ est divisible henselien et $v(N)=v(M)$, donc $(N,v)$ est $C$-minimal par la proposition précédente. Alors $M=\bigcup_{a \in M_{tor}} M_{>\theta} + a$, est une union finie de boules de $N$, par conséquent il est $C$-minimal.
\end{proof}
\end{lem}

Une conséquence immédiate des deux résultats précédents est:
\begin{corr}\label{puiseuxCmin}
L'anneau de valuation d'un corps de caractéristique $p>0$, parfait, muni d'une valuation henselienne,  de corps résiduel fini ou p-clos, et de groupe de valuation divisible, est {C}-minimal en tant que $\F_p[t;x\to x^p]$-module valué.
\begin{proof}
Un groupe abélien ordonné divisible est $o$-minimal, donc a fortiori sa restriction à la structure de  $\F_p[t;x\to x^p]$-chaîne. 
Le corollaire découle donc du 
 lemme 3.24 \cite{Onay2017}, avec la proposition \ref{divhens->c-min} si le corps résiduel est $p$-clos et avec le lemme  \ref{puiseux-stayla-c_min}
si le corps résiduel est fini.
\end{proof}
\end{corr}

\begin{exemple}
Soit $R=\F_p[t;x\mapsto x^p]$.
L'anneau des séries de Puiseux sur $\F_p$ ou sur un corps $p$-clos est $C$-minimal en tant que $R$-module valué, mais le corps des séries de Puiseux ne l'est pas (en tant que $R$-module). 
\end{exemple}

Le lemme ci-dessous, avec le lemme précédent, montreront que la condition $2.$ cité dans le théorème \ref{thmCmintrivi}, implique la $C$-minimalité.

\begin{lem}
Soit $(M,v)$  un module $K$-trivialement valué henselien et tel que
\begin{enumerate}
\item $M_{tor}$ est fini  et $M=M_{tor} \oplus M_{>\theta}$,
\item $M_{>\theta}.t$ contient une boule  de la forme $M_{\geq \gamma}$ ou $M_{>\gamma}$,  qui est d'indice fini dans $M_{>\theta}.t$   et $v(M_{>\theta})$ est $o$-minimal.
\end{enumerate}
Alors $(M,v)$ est $C$-minimal.
\begin{proof}
Si $M_{>\theta}$ est divisible, alors l'assertion suit de lemme ci-dessus. 

Montrons d'abord que la valuation sur $M_{>\theta}$  détermine une et une seule valuation sur chacune de ses clôtures divisibles (toutes sont d\'ej\`a isomorphes en tant que $R$-modules). 
Soit $N$ une clôture divisible   de $M_{>\theta}$. Alors, puisque $M_{>\theta}$ est divisible par tous les polynômes séparables, 
 $x \in N$ si et seulement si  $x.t^n \in M_{>\theta}$ pour un certain $n \in N$.  Soit alors $x \in N\setminus M_{>\theta}$ et $n$ minimal tel que $x.t^n \in M_{>\theta}$. Soit  $\gamma=v(x.t^n)$. On définit $\bar{v}(x)$ comme le couple $(\gamma, n)$. Si $y \in N\setminus M_{>\theta}$ est tel que $\bar{v}(y)=(\delta, k)$, on décrète que 
$\bar{v}(x) < \bar{v}(y)$ si et seulement si $\tau^k(\gamma) < \tau^n(\delta)$. Il est immédiat de vérifier que $(N,\bar{v})$ est un module $K$-trivialement valué, et que $\bar{v}$ prolonge $v$. 

Montrons maintenant que $M_{>\theta}$ s'identifie à une union finie de boules de $(N,\bar{v})$. Soit $B$ la boule de $M$,  d'indice fini dans $M_{>\theta}.t$, dont l'existence est donnée par l'hypothèse du corollaire. Montrons que  $B$ est en fait une boule de $N$.  Soit  $\gamma$ le rayon de $B$, i.e. $B=M_{>\gamma}$ ou $B=M_{\geq \gamma}$. Soit $x \in N \setminus M_{>\theta}$ , avec $v(x) > \gamma$, et soit $k\geq 1$ minimal tel que $\overline{x}.t^k \in  M_{>\theta}$. On a, $v(x.t^k) > 
\tau^{k}(\gamma)> \gamma$. Or $M_{>\theta}.t \supset B$ implique $x.t^{k-1} \in M_{>\theta}$. Contradiction. D'où $B=N_{\geq \gamma}$ ou
$B=N_{>\gamma}$. En écrivant $M_{>\theta}.t= \bigcup_i B +a_i$, on a $M_{>\theta}=B.t^{-1} + a_.t^{-1}$, où $B.t^{-1}$ est nécessairement la boule ouverte ou fermée de rayon $\tau^{-1}(\gamma)$ et $a_i.t^{-1}$ est l'unique $b_i \in N$ tel que $b_i.t=a_i$. Par conséquent, $M_{>\theta}$   est une union finie de boules de $N$ donc $M$ est une union finie de boules de $N\oplus M_{tor}$ et il suffit de montrer que $N\oplus M_{tor}$ est $C$-minimal.  
Puisque $(N, \overline{v})$ est divisible henselien sans-torsion, le lemme \ref{puiseux-stayla-c_min} dit  qu'il suffit de montrer que 
$\overline{v}(N)$ est $o$-minimal.
 
Montrons  d'abord que  $\overline{v}(N)$
est la clôture de $v(B)$ par $\tau^{-1}$.
 Soit  $x \in M_{>\theta}$, tel que $x.t \notin B$. On a $v(x.t-a_i)>v(x.t)=v(a_i)$ pour un certain $i \in \{1,\dots, n\}$. Dans ce cas , $v(x.t^2)>v(x.t)$ et donc pour tout $i$, $v(x.t^2-a_i)=v(a_i)$. Alors $x.t^2 \in B$ nécessairement. Par ailleurs, si $x.t \in B$ de même pour  $x.t^2$. Donc $N$ est aussi la clôture divisible de $B$. 
D'où $\overline{v}(N)$
est la clôture de $v(B)$ par $\tau^{-1}$,  $\overline{v}(N)=\bigcup_{i \in \N } \tau^{-i}v(B)$, et chaque   $\tau^{-i}v(B)$ est isomorphe à $v(B)$ donc ils sont  tous $o$-minimaux. 

Puisque $v(M)$ est $o$-minimal $v(M)\setminus\{\theta,\infty\}$ est dense ou discret: Soit  $DE$ l'ensemble des points  ayant un voisinage dense et $DI$ l'ensemble des points ayant un prédécesseur et un successeur; par $o$-minimalité, $v(M)$ se partionne comme $DE\cup DI\cup F$, avec $F$ fini, plus précisement $DE$ et $DI$ sont des unions finis d'intervalles séparés par au moins un point de $F$; par ce que $\tau$ est strictement croissant, il préserve chacun des trois ensembles, $DE$, $DI$ et $F$, et parce qu'il n'admet aucun point fixe sur $v(M)\setminus\{\theta,\infty\}$,  $F\setminus\{\theta,\infty\}=\emptyset$, par conséquent $v(M)\setminus\{\theta,\infty\}=DE$ ou $v(M)\setminus\{\theta,\infty\}=DI$.

Maintenant, par \cite{Pillay1987}, si $v(M)$ est discret alors $\tau$ est une translation sur $v(M)$ et donc $\overline{v}(N)$ est discret et $\tau$ reste une translation. Sinon si $v(M)$ est dense, par construction $\overline{v}(N)$ est dense. Dans le 
deux cas, la $\overline{v}(N)$ est $o$-minimal par le corollaire 1.16 de \cite{Maalouf2010a}.
\end{proof}
\end{lem}

Les lemmes qui suivent   établissent des réciproques  aux résultat ci-dessus et ainsi complètent la preuve du théorème \ref{thmCmintrivi}.
\begin{lem}\label{c-min2div}
 Soit $(M,v)$  un module $K$-trivialement valué, ${ C}$-minimal, infini.
S'il existe  $a \in M$ tel que $v(a)< \theta$   alors M est divisible.
\begin{proof}
Soit $q\in R$, non nul. S'il existe un $a \in M$ comme ci-dessus, alors $v(M.q)$
n'a pas de premier élément. Si $M.q \neq M$ alors par la proposition \ref{c-min_d'indice_fini} il contient un sous-groupe $H$ de la forme
$M_{\geq \gamma}$ ou de la forme $M_{>\gamma}$ qui est d'indice fini dans $M.q$. Alors $M.q$ est la réunion disjointe finie des $H + a_i$ et donc $v(M.q)$ a un  premier élément. Contradiction.  
\end{proof}
\end{lem}

\begin{lem}\label{anninfini_implique_Cmin}
Soit $(M,v)$ un module non trivialement valué qui est $K$-trivialement valué et $C$-minimal. Alors, pour tout $r \in R\setminus\{0\}$, $\text{ann}_M(r)$   est fini.
\begin{proof}
Si $\text{ann}_M(r)$,  est une partie propre infini de $M$, alors par la proposition 
\ref{c-min_d'indice_fini} il existe $\gamma \in \Delta$ tel que $\text{ann}_M(r)$  contient une boule de la forme $M_{\geq \gamma}$ ou $M_{>\gamma}$ qui est d'indice fini dans $\text{ann}_M(r)$. Comme $v(\text{ann}_M(r)) = \{\theta\}$, on obtient  $\gamma=\theta$; dans ce cas là, on a nécessairement $M_{>\theta}={0}$ et donc $\text{ann}_M(r)=M_{\geq\theta}=M_{tor}$. Cela implique en particulier $\infty$ est l'unique élément qui est strictement$>\theta$. Puisque $(M,v)$ est non trivialement valué, il existe donc des éléments de valuation $<\theta$. Par le lemme ci-dessus $M$ est alors divisible; d'où $M_{tor}$ est divisible.  Or ceci contredit le fait que $M_{tor}$ est annulé par un seul élément.
\end{proof}
\end{lem}

\begin{corr}
Soit $(M,v)$ non trivialement valué, $K$-trivialement valué et $C$-minimal. Si $M_{tor}$ est infini alors $M_{tor}$ est divisible et fortement minimal.
\end{corr}
\begin{proof} $M_{tor}$ se plonge dans  le $R$-module $M_{\geq \theta}/M_{>\theta}$ qui est fortement minimal, on a la divisibilité de $M_{\geq \theta}/M_{>\theta}$ par le lemme ci-dessus et par le lemme \ref{frtmin}, d'où la divisibilité de $M_{tor}$. Alors par la proposition 2.15 dans \cite{Onay2017}, tout sous-ensemble définissable de $M_{tor}$ est une combinaison booléenne  d'ensembles du type $\text{ann}_M(r)$, avec $r \in R\setminus\{0\}$. D'où, par le théorème \ref{frtmin} et par le fait
	que chaque annulateur est fini, la forte minimalité de $M_{tor}.$ 
\end{proof}

\begin{lem}
Soit $(M,v)$  un module $K$-trivialement valué ${C}$-minimal. Alors $(M,v)$ est henselien.
\begin{proof}
On va considérer $2$ cas:\\
{\bf a.} {\it $v(M_{>\theta})$ n'a pas de premier élément.} 
Par la proposition
\ref{c-min_d'indice_fini}, si  $M_{>\theta}.r$ ($r \in R\setminus 0$) n'est pas égal à $M_{>\theta}$, alors il contiendrait 
une boule ouverte ou fermée d'indice fini. 
En particulier, ou bien $v(M_{>\theta }.r)$ aurait un plus petit élément ou bien pour un certain $\gamma>\theta$ on aurait 
$M_{>\theta }.r=M_{>\gamma}$. Ces deux cas sont impossibles car si $r$ est séparable et $x \in M_{>\theta}$, on a  $v(x.r)=v(x)$.

\noindent
{\bf b.} {\it $v(M_{>\theta})$ a un plus petit élément.} 
Soit $\gamma_{0}$ ce plus petit élément $> \theta$. Soit $s \in R_{sep}$. 
Écrivons $s$ comme $tq+ a$ avec $a \in K$. Encore par le lemme 
\ref{c-min_d'indice_fini}, si  $M_{>\theta}.s \neq M_{>\theta}$ alors il contient
un certain $M_{>\gamma}$ ou $M_{\geq \gamma}$ qui est d'indice fini dans $M_{>\theta}.s$. Supposons que c'est $M_{> \gamma}$, l'autre cas se traite de même. Puisque $M_{>\gamma}$ est d'indice fini dans $M_{>\theta}.s$, il existe  $\gamma_1, \dots \gamma_k$ tel que $\gamma_0 < \gamma_1 < \dots < \gamma_k= \gamma$ avec $\gamma_{i+1}$  le successeur de $\gamma_{i}$ dans $v(M)$. Soit $x$ de valuation $\gamma_0$, alors 
$x_0:=(x - x.a^{-1}(tq + a))=(x - x.a^{-1}s)$ est de valuation $ \geq \gamma_1 >\gamma_{0}$. Ainsi
on définit $x_i:=(x_{i-1} - x_{i-1}.a^{-1}s)$, avec $v(x_i) > \gamma_i$, pour $1\leq i \leq k$. Donc $v(x_{k})>\gamma$. Par conséquent $x_k$ est divisible
par $s$, et donc $x$ aussi. On a montré que $M_{>\theta}$ est divisible
par les polynômes séparables, i.e. $(M,v)$ est henselien.
\end{proof}
\end{lem}

%

Par ce qui précède, on peut supposer que $M=M_{\geq \theta}$ et que $M_{tor}$ est fini. Il nous  reste alors à montrer le résultat suivant, avec lequel s'achève la preuve du théorème \ref{thmCmintrivi}.

\begin{lem}
Soit $(M,v)$ un module $K$-trivialement valué $C$-minimal non divisible et tel que 
$M_{\geq \theta}=M$ et $M_{tor}$ est fini. Alors $M=M_{tor} \oplus M_{>\theta}$.
\begin{proof}
Puisque $M_{tor}$ est fini la multiplication par $t$ induit une bijection
de $M_{tor}$. Donc $M_{tor}=\text{ann}_M(t^k-1)$ pour un certain $k \in \N\setminus \{0\}$. Par conséquent, $M_{tor}$ n'est pas divisible donc $M$ non plus. Donc $M.(t^k-1)$ contient une boule (ouverte ou fermée) qui y est d'indice fini, d'après le lemme \ref{c-min_d'indice_fini}. En particulier le quotient 
$M.(t^k-1)/(M_{>\theta} \cap M.(t^k - 
1))$ est fini; de plus il est sans torsion. Or, puisque $R$ est infini, le 
seul module sans torsion fini est le module $0$. Cela implique  que 
$M.(t^k-1) \subset M_{>\theta}$. 

Montrons  l'inclusion réciproque $M_{>\theta}\subset M.(t^k-1)$: $(M,v)$ est henselien par le lemme ci-dessus; en particulier $M_{>\theta}$ est divisible par le polynôme $t^k-1$; cela implique  que la suite  ci-dessous est exacte (et clairement scindée):
$$0 \longrightarrow M_{tor} \longrightarrow M \xrightarrow{.(t^k-1)} (M_{>\theta})  \longrightarrow 0.$$
D'où $M_{>\theta}$ est facteur direct dans $M$.
\end{proof}
\end{lem}

\section{Cas d'une valuation non $K$-triviale}
Les propositions ci-dessous établissent la preuve du théorème \ref{thmKnontrivi}. Rappelons que $(M,v)$ est dit r\'esiduellement divisible, si pour tout non zero $r\in R$, non zero $z\in M$ il existe $y\in M$, tel que 
$$v(z-y.r)>v(y.r)=v(y)\cdot r.$$

Notons que l'\'egalit\'e dans l'expression ci-dessus  peut ne peut \^etre toujours v\'erifi\'ee (on peut penser au cas o\`u $y$ est non zero et annul\'e par $r$.)

\begin{prop}\label{cmin}
Un module  valué $(M,v)$ affinement maximal, résiduellement divisible, ayant sa chaîne $v(M)$ $o$-minimale et tel que, pour chaque $r \in R\setminus\{0\}$, $\text{ann}_M(r)$ est fini, est $C$-minimal.
\begin{proof}
Par \cite{Onay2018a}, Th\'eor\`eme 4.6, on sait que toute formule à une seule variable libre $x$, avec paramètres, est équivalente à une formule de la forme $$\phi(x) \wedge \psi(v(t_1(x)),\dots,v(t_k(x)), \overline{\gamma}),$$ où 
$\phi$ est une formule sans quantificateur avec paramètres dans le langage des $R$-modules, $\psi$ est
une formule sans paramètres de $L_V'$,  les $t_i$ sont des termes avec paramètres dans le langage des $R$-modules  et $\overline{\gamma}$ est un $n$-uplet d'éléments de $v(M)$.
En particulier, $\phi$ est une combinaison booléenne de formules de la forme $x.r=b$. Puisque $\text{ann}_M(r)$ est fini, elle définit un ensemble fini ou cofini. On se ramène donc à montrer que $\psi(v(t_i(x))_{i})$ décrit une combinaison booléenne de boules ouvertes ou fermées.
On se contentera de montrer que l'on peut remplacer les termes $v(t_i(x))$ par des termes
de la forme $\tau^k(v(x-c_i))$ où $k \in \N$ et $c_i \in M$. Ainsi $\psi$ sera de la forme $\psi'(v(x-c_1),\dots,v(x-c_n))$. 
La $C$-minimalité de $(M,v)$ suit alors de l'$o$-minimalité de la $R$-chaîne $v(M)$   par des considérations ultramétrique générales (comme il a été explicité dans la preuve de la proposition 36 de \cite{Maalouf2010a}). 

\noindent
Les termes $t_i(x)$ sont de la forme $x.r_i - a_i$ ($r_i \in R$, $a_i \in M$). 
Par divisiblité de $M$, il existe $b_i \in M$ tel que 
$b_i.r_i= a_i$. Ainsi chaque terme $t_i(x)$ est égal à un terme de la forme $(x-b_i).r_i$. La preuve sera donc achevée dès que nous aurons prouvé l'assertion suivante:

\noindent
Etant donné $r \in R\setminus\{0\}$, il existe un recouvrement fini de $M$ par des ensembles dont chacun, disons $E$, est une combinaison booléenne de boules de $M$ et tel que pour un certain paramètre $d$,  
$(M,v) \models \forall x \in E, 
           v(x  \cdot r ) = v(x-d) \cdot r$. 

\noindent
En effet, d'une part par le résultat 3.36 dans \cite{Onay2018a}, on a 
$$\forall x \bigvee_{c \in ann_M(r)} v(x.r)=v((x-c).r)=v(x-c)\cdot r.$$
Et d'autre  part, pour tout $c_0 \in \text{ann}_M(r)$, l'ensemble 
$\{x \in M \tq v((x-c_0).r)=v(x-c_0)\cdot r\}$ est exactement l'ensemble des éléments $x$ tel que $x-c_0$ est régulier pour $r$,  et est égal, par le m\^eme r\'esultat, l'ensemble $\{x \tq v(x-c_0)=\max\{v(x-c_0-c); c\in \text{ann}_M(r) \}\}$. Ce dernier ensemble est clairement une combinaison booléenne de boules.

\end{proof}
\end{prop}

Pour la d\'emonstration suivante on rappelle que l'ensemble $Saut_M(r)$ (not\'e $\jumpm{M}{r}$ dans \cite{Onay2018a}) est l'ensemble des valeurs $\gamma$ tel qu'il existe un $x\in M$ de valuation $\gamma$ satisfaisant 
$$v(x.r)>v(x)\cdot r.$$

\begin{prop}
Soit $(M,v)$ un module valué non $K$-trivialement  et $C$-minimal. Alors
$(M,v)$ est affinement maximal et résiduellement divisible, et pour tout 
$r\in R\setminus \{0\}$, 
$\text{ann}_M(r)$ est fini.
\begin{proof}
Puisque $v_K$ est non triviale sur $K$,  $v(M)$ n'est pas minoré. Or 
$Saut_M(r)$ est fini pour tout $r \in R \setminus \{0\}$ donc $\text{ann}_M(r)$
ne peut pas contenir de boule propre d'indice fini et doit être fini. 
De plus, pour les mêmes raisons, $v(M.r)$ n'est pas minoré, donc $M.r$ 
ne peut contenir de boule propre d'indice fini et $M.r=M$. D'où $M$ est divisible. 

Supposons que $(M,v)$ n'est pas affinement maximal. Comme $M$ est d\'ej\`a divisible, il est r\'esiduellement divisible. Dans ce cas, par le théorème 3.35 \cite{Onay2018a},
  il existe $r \in R \setminus\{0\}$ et $y \in M\setminus\{0\}$ tels que pour tout $x$ vérifiant $x.r=y$, $x$ est irrégulier pour $r$. Puisque 
$M$ est divisible, l'ensemble $A:=\{x \in M \tq x.r=y\}$ est non-vide et puisque 
$\text{ann}_M(r)$ est fini, il s'écrit comme $A=\{x_1,\dots, x_k\}$. Posons
$\gamma=\max_i\{v(x_i)\}$ et prenons $x\in A$ de valuation $\gamma$. Alors
$y \in M_{>\gamma \cdot r} \setminus \left(M_{>\gamma}\right).r$. On va montrer que ceci conduit à une contradiction. Puisque $\gamma$ est un saut, il est limite inférieure dans $v(M)$  des éléments $\delta \in v(M)\setminus Saut_M(r)$ par le lemme 2.11 \cite{Onay2018a} et par la finitude de $Saut_M(r)$; donc il est limite des éléments $\delta$ tel que  $v(M_{>\delta}.r)$ et $v(M_{>\delta \cdot r})$ définissent le même segment final de $v(M)$. Cela impose que 
 $M_{>\gamma \cdot r}=\left(M_{>\gamma}\right).r$. Contradiction.
 \end{proof}
\end{prop}

Enfin, on donne une conséquence du résultat ci-dessus dans le cas des corps valués: par  le lemme 2.11 \cite{Onay2018a}, il suit qu'une $R$-chaîne $o$-minimale est dense et pleine et il suit par le corollaire \cite{Onay2018a} 2.26, qu'une $R$-cha\^ine pleine et dense est $o$-minimale. De plus,
  par la remarque 2.29 \cite{Onay2018a},  un groupe abélien ordonné plein comme $R$-chaîne est divisible. Ensuite, par l'exemple 2.28(2) \cite{Onay2018a},  la $R$-cha\^ine $v(\mathcal{K})$, d'un corps aux difference valu\'e $(\mathcal{K},\sigma, v)$, o\`u $R=\mathcal{K}[t;\sigma]$ est pleine et dense que $v(\mathcal{K})$ est un $\mathbb{Z}[\sigma]$-module (de nouveau par \cite{Onay2018a}, 2.28). Enfin par le corollaire \cite{Onay2018a}, un corps   valué hens\'elien $(K,v)$ de caractéristique $p>0$ est algébriquement maximal si et seulement s'il est affinement maximal comme $K[t;x\mapsto x^p]$-module valué. 

\begin{corr} On a:
\begin{itemize}
	
\item Soit $(K,v)$  un corps   valué infini hens\'elien et de caractéristique $p>0$. 
 Alors $(K,v)$  est $C$-minimal comme $K[t;x\mapsto x^p]$-module valué si et seulement s'il est algébriquement maximal avec  un corps résiduel $p$-clos
et un groupe de valuation divisible.

\item Soit $\mathcal{K}$ l'ultraproduit  des corps alg\`ebriquement valu\'es $\mathcal{K}_{p^n}$,  de caract\'eristique  $p>0$, \'equipp\'e chacun du morphisme $x\mapsto x^{p^n}$, selon un ultrafiltre $U$ sur $\{p^n \tq n\in \mathbb{N}, \; \text{et} \; p \; \text{premier}\}$. $\mathcal{K}_{p^n}$   muni de l'automorphisme  limite, {\it Frobenius non-standard}, i.e., $\sigma_{U}:=\lim_{U} x \mapsto x^{p^n}$, est $C$-minimal comme $\mathcal{K}[t;\sigma]$-module valu\'e.
\end{itemize}
\end{corr}
\begin{remarque}
Noter que en prenant un ultrafiltre trivial, le second point ci-dessus implique le premir.
\end{remarque}
\bibliographystyle{apalike}

\bibliography{/home/onayg/onayg.bib}
\end{document}